\documentclass{article}
\makeatletter
\newcommand\footnotetext@\relax
\let\footnotetext@\@footnotetext
\renewcommand{\thanks}[1]{%
  \unskip\protected@xdef\@thefnmark{}%
  \protected@xdef\@thanks{\@thanks\protect\footnotetext@{#1}}}
\newcommand{\MSC}[2][2010]{%
  \unskip\protected@xdef\@thefnmark{}%
  \protect\footnotetext@{\kern-1.8em{\it MSC (#1):} #2.}}
\newcommand{\keywords}[1]{%
  \unskip\protected@xdef\@thefnmark{}%
  \protect\footnotetext@{\kern-1.8em{\it Keywords:} #1.}}
\newcommand{\address}[1]{\\[.67ex]%
  \parbox{.95\textwidth}{\centering\footnotesize\itshape#1}}
\gdef\@date{}
\makeatother
\usepackage[english]{babel}
\usepackage{amsmath}%
  \allowdisplaybreaks[4]
\usepackage{amsthm}%
  \theoremstyle{plain}
    \newtheorem{thm}{Theorem}
    \newtheorem{lem}[thm]{Lemma}
    \newtheorem{prop}[thm]{Proposition}
    \newtheorem{cor}[thm]{Corollary}
  \theoremstyle{definition}
    
\usepackage[all]{xy}%
  \entrymodifiers={!!<.em,.6ex>+}
  \SelectTips{cm}{}
\usepackage{amssymb}
\usepackage{mathrsfs}
\usepackage{hyperref}
\makeatletter
\newcommand{\aeq}{\Leftrightarrow}
\newcommand{\tto}{\rightrightarrows}
\newcommand{\xto}{\xrightarrow}
\newcommand{\xfro}{\xleftarrow}
\newcommand{\longto}{\longrightarrow}
\newcommand{\Z}{\bgroup\mathbb Z\egroup}
\newcommand{\R}{\bgroup\mathbb R\egroup}
\newcommand{\C}{\bgroup\mathbb C\egroup}
\newcommand{\T}{\bgroup\mathbb T\egroup}
\DeclareMathOperator{\End}{End}
\DeclareMathOperator{\Hom}{Hom}
\DeclareMathOperator{\supp}{supp}
\newcommand{\GL}{\operatorname{GL}}
\newcommand{\GU}{\operatorname{U}}
\newcommand{\SU}{\operatorname{SU}}
\newcommand\@d{\mathopen{\mathit d}}
\newcommand\@@d{\mkern\thinmuskip\@d}
\renewcommand{\d}{\@d}
\DeclareSymbolFont{wasyfonts}{U}{wasy}{m}{n}
\DeclareMathSymbol{\varint}{\mathop}{wasyfonts}{"72}
\newcommand\@integral\relax
\newcommand\beforestar@d\relax
\newcommand\afterstar@d\relax
\newcommand\@unstar@d\relax
\def\@unstar@d#1*d#2#3{%
   \def\beforestar@d{#1}%
  \edef\afterstar@d{#2}}
\def\@integral#1#2#3#4d{%
  \@unstar@d #4d*d\@empty\@empty
   #1\varint #2%
  \ifx\afterstar@d\@empty #4#3\@@d
  \else\beforestar@d #3\@d\fi}
\newcommand\@plain@integral{%
  \@integral{}{}{}}
\newcommand\@star@integral\relax
\def\@star@integral#1d{%
  \mathchoice{%
    \@integral\textstyle\limits\displaystyle#1d}%
   {\@integral{}\limits{}#1d}%
   {\@integral{}\limits{}#1d}%
   {\@integral{}\limits{}#1d}}
\newcommand{\integral}{%
  \@ifstar\@star@integral\@plain@integral}
\newcommand{\catfont}{\mathscr}
\newcommand{\catname}[1]{\underline{\smash[b]{\mathbf{#1}}}}
\newcommand{\Rep}{\catname{Rep}}
\newcommand{\HRep}{\catname{HRep}}
\newcommand{\DRep}{\catname{DRep}}
\DeclareMathOperator{\Zcat}{Z}
\DeclareMathOperator{\Hcat}{H}
\makeatother
\title{Some remarks on representations up to homotopy%
  \MSC{22A22, 18G35, 18G55}
  \keywords{groupoids, representations up to homotopy, Peter--Weyl theorem}
}%
\author{Giorgio Trentinaglia%
  \address{Center for Mathematical Analysis, Geometry and Dynamical Systems, %
          Instituto Superior T\'ecnico, Universidade de Lisboa, %
          Av.~Rovisco Pais, 1049-001 Lisboa, Portugal}
   \and Chenchang Zhu%
  \address{Mathmatics Insitute, %
          Georg-August-Universit\"at G\"ottingen, %
          Bunsenstrasse 3-5, 37073, G\"ottingen, Germany}
  \thanks{G.~Trentinaglia acknowledges support %
         from the Portuguese Foundation for Science and Technology %
         (Fun\-da\-\c c\~ao pa\-ra a Ci\-\^en\-cia e a Tec\-no\-lo\-gia) %
         through postdoctoral grant %
         SFRH/\allowbreak BPD/\allowbreak 81810/\allowbreak 2011 %
         and, partly, through grant %
         UID/\allowbreak MAT/\allowbreak 04459/\allowbreak 2013. %
         C.~Zhu acknowledges support %
         from the German Research Foundation %
         (DFG, Deut\-sche For\-schungs\-ge\-mein\-schaft) %
         through the Institutional Strategy of the University of G\"ot\-tin\-gen.}
}%
\begin{document}
\maketitle

\begin{abstract} Motivated by the study of the interrelation between functorial and algebraic quantum field theory, we point out that on any locally trivial bundle of compact groups, representations up to homotopy are enough to separate points by means of the associated representations in cohomology. Furthermore, we observe that the derived representation category of any compact group is equivalent to the category of ordinary (finite-dimensional) representations of the group. \end{abstract}

\subsection*{Introduction}

It is well known to mathematical physicists that there is a systematic way of assigning an algebraic quantum field theory (AQFT) to a ``functorial'' quantum field theory (FQFT) \cite{Sch09}. The present paper collects some observations which are potentially relevant to the converse problem, namely, that of recovering a FQFT from the corresponding AQFT, in the spirit of Doplicher and Roberts' reconstruction theorem \cite{DR89a,DR89b,HM06}. The point we want to make here is that the notion of `representation up to homotopy', recently introduced by differential geometers \cite{AC13,AC12,ACD11}, could play a role in the study of such problem.

From a mathematical physicist's perspective, one compelling reason to consider representations up to homotopy (hereafter named `homotopy representations') is that they occur naturally as representations of higher Lie groups. Higher symmetry groups are ubiquitous in quantum field theory \cite{Sha15}. There are also situations in quantum field theory where the gauge symmetries of the local fields may be expected to constitute a groupoid, rather than a group; specifically, we have in mind the example of 2-C*-categories with non-simple units \cite{Zit07}, or quantum field theories in curved spacetimes \cite{FV,Few15}. With respect to groupoids more general than groups, the notion of homotopy representation appears to be better behaved than other notions of representation, as we shall have occasion to point out in the course of the paper.

A basic question is whether the classical Peter--Weyl theorem for compact group representations also holds for homotopy representations of more general proper groupoids. The relevance of this question from the point of view of the above-mentioned reconstruction problem is clear: the Doplicher--Roberts method shows that from the AQFT associated to a FQFT with gauge group $G$ one can extract a tensor category which is equivalent to the category of finite-dimensional representations of $G$, and then the Peter--Weyl theorem guarantees that $G$ can be recovered from that tensor category\textemdash so that no information concerning the unobservable gauge symmetries of the FQFT is lost when passing to the corresponding AQFT. Now, the naive generalization of the theorem, phrased only in terms of ordinary groupoid representations on vector bundles, is known to fail already for locally trivial bundles of compact abelian groups \cite{Tre10}. In this article, after explaining in what sense homotopy representations may be ``enough'' to separate points, we show that such representations (even just the \emph{strongly invertible} ones, i.e.~those whose underlying pseudo-representations are invertible) are indeed enough to separate points on any locally trivial bundle of compact abelian groups. (Notice that the adjoint representation of any such group bundle is trivial and thus useless.) More in general, we show that any locally trivial bundle of compact groups admits enough homotopy representations (perhaps not strongly invertible). We also highlight a topological obstruction to extending a given isotropy representation to an invertible globally defined pseudo-representation in a specific example.

In the second part of the paper, we observe that for any compact group $G$ the category $\Rep(G)$ of ordinary (finite-dimensional) representations of $G$ can be recovered from the differential graded category of homotopy representations of $G$ by means of a standard, purely categorical, procedure. This result ensures, a~priori, that the classical reconstruction theorem for compact groups can be stated alternatively in terms of homotopy representations, as it should if the approach that we foresee is correct, and also provides hints as to how to proceed in more general situations.

\paragraph*{Overall conventions.} Throughout the paper, the term `groupoid' will be synonymous with `locally compact (topological) groupoid'. By a \emph{proper} groupoid, we shall mean a locally compact Hausdorff groupoid with proper anchor map that also admits a normalized Haar measure system \cite{Tu99}. All graded vector bundles $E^\bullet=\bigoplus E^{(n)}$ will be $\Z$-graded and of finite rank, in particular, the condition $E^{(n)}=0$ will be satisfied for $|n|\gg0$. Given any pair of graded vector bundles $E^\bullet$ and $F^\bullet$ and any integer $n$, we shall let $\Hom^{(n)}(E^\bullet,F^\bullet)$ denote the vector bundle of all linear maps of degree $n$ between the fibers of $E^\bullet$ and $F^\bullet$.

\subsection*{Homotopy representations of groupoids}

Let $G$ be a groupoid, $G=G_1\tto G_0$. For each integer $k\geq2$, let $G_k$ denote the space of composable $k$-tuples of arrows in $G$:
\[
 G_k=\{(g_1,\dotsc,g_k)\in G\times\dotsb\times G\mathrel|sg_j=tg_{j+1}\;\forall j\leq k-1\}.
\]
Let $G_0\xfro{t_k}G_k\xto{s_k}G_0$ denote the two maps $(g_1,\dotsc,g_k)\mapsto tg_1$ and $\mathellipsis\mapsto sg_k$. We agree that $s_0=t_0=id$ (on $G_0$) for $k=0$. By a \emph{pseudo-representation} $\lambda$ of $G$ on a vector bundle $E$ over $G_0$ (graded or not) we shall mean a morphism of vector bundles $\lambda:s^*E\to t^*E$.

Recall that a `representation up to homotopy', hereafter called \emph{homotopy representation}, of $G$ \cite[Section~3.1]{AC13} consists of a graded vector bundle $E^\bullet$ (of finite rank) over $G_0$ and of a sequence $\{R_k\}_{k\geq0}$ of continuous sections \[R_k\in\Gamma\bigl(G_k;\Hom^{(1-k)}(s_k^*E^\bullet,t_k^*E^\bullet)\bigr)\] which satisfy the condition
\begin{multline}
\label{eqn:2011a/1}
 \sum_{j=1}^{k-1}(-1)^jR_{k-1}(g_1,\dotsc,g_jg_{j+1},\dotsc,g_k)=%
\\[-.5\baselineskip]%
 \sum_{j=0}^k(-1)^jR_j(g_1,\dotsc,g_j)\circ R_{k-j}(g_{j+1},\dotsc,g_k)
\end{multline}
for every $k\geq0$. We shall say that a homotopy representation is \emph{strongly invertible} if the (degree zero) morphism of graded vector bundles $R_1:s^*E^\bullet\to t^*E^\bullet$ is an invertible pseudo-representation, that is, an isomorphism of vector bundles over $G_1$. The notion of strong invertibility appears to be relevant especially in connection with the behavior of conjugate (or dual) representations.

Let us rename $R_0,R_1,R_2$ respectively $\partial,\lambda,\Psi$. Such auxiliary notations will make the intuitive content of the next equations clearer.
\[\xymatrix{%
 \cdots \ar@{}[d]|*{\cdots}
        \ar[r]
	& E^{(-1)} \ar[d]^(.33){\lambda^{(-1)}}
	           \ar[dl]
	           \ar[r]^-{\partial^{(-1)}}
		& E^{(0)} \ar[d]^(.33){\lambda^{(0)}}
		          \ar[dl]|-*+<.67ex>{\scriptstyle\Psi^{(0)}}
		          \ar[r]^-{\partial^{(0)}}
			& E^{(1)} \ar[d]^(.33){\lambda^{(1)}}
			          \ar[dl]|-*+<.67ex>{\scriptstyle\Psi^{(1)}}
			          \ar[r]^-{\partial^{(1)}}
				& E^{(2)} \ar[d]^(.33){\lambda^{(2)}}
				          \ar[dl]|-*+<.67ex>{\scriptstyle\Psi^{(2)}}
				          \ar[r]
					& \cdots \ar@{}[d]|*{\cdots}
					         \ar[dl]
\\
 \cdots \ar[r]
	& E^{(-1)} \ar[r]^-{\partial^{(-1)}}
		& E^{(0)} \ar[r]^-{\partial^{(0)}}
			& E^{(1)} \ar[r]^-{\partial^{(1)}}
				& E^{(2)} \ar[r]
					& \cdots
}\] In the new notations, the equations \eqref{eqn:2011a/1} for $k=0,1,2$ read as follows.
\begin{flalign}
\label{eqn:2011a/2}
 k&=0\text:%
&\makebox[.em][c]{$\partial\circ\partial=0$}
\\%
\label{eqn:2011a/3}
 k&=1\text:%
&\makebox[.em][c]{$\partial_{tg}\circ\lambda(g)=\lambda(g)\circ\partial_{sg}$}
\\%
\label{eqn:2011a/4}
 k&=2\text:%
&\makebox[.em][c]{$\lambda(g)\circ\lambda(h)-\lambda(gh)=\partial_{tg}\circ\Psi(g,h)+\Psi(g,h)\circ\partial_{sh}$}
&&
\end{flalign}
Equation~\eqref{eqn:2011a/2} expresses the condition that $\partial$ is a differential of degree $+1$ on the graded vector bundle $E^\bullet$. At each base point $x\in G_0$ one gets a cochain complex of vector spaces $(E^\bullet_x,\partial_x)$, with corresponding cohomology groups $H^i(E^\bullet_x,\partial_x)$. These cohomology groups fit together into a \emph{cohomology bundle} $H^i(E^\bullet,\partial)$, which is to be understood just as an abstract family of vector spaces indexed over $G_0$. The dimensions of the vector spaces $H^i(E^\bullet_x,\partial_x)$ may jump discontinuously across $G_0$. Equation~\eqref{eqn:2011a/3} implies that each linear map $\lambda(g):E^\bullet_{sg}\to E^\bullet_{tg}$ induces a well-defined linear map $H^i(E^\bullet_{sg},\partial_{sg})\to H^i(E^\bullet_{tg},\partial_{tg})$ in cohomology. Equation~\eqref{eqn:2011a/4} then implies that the latter linear maps form a ``representation'' of the groupoid $G$ on the cohomology bundle $H^i(E^\bullet,\partial)$, although only in a weak sense: groupoid units need not map to identity linear maps.

A strongly invertible homotopy representation is automatically \emph{unital up to homotopy}, in the following sense: because of the invertibility of the associated pseudo-representation, for each base point $x\in G_0$ we can set
\begin{equation*}
 \Upsilon(x)=\lambda(1_x)^{-1}\circ\Psi(1_x,1_x)
\end{equation*}
and then deduce the following identity from \eqref{eqn:2011a/4} [and \eqref{eqn:2011a/3}]
\begin{equation}
\label{eqn:2011a/5}
 \lambda(1_x)-id_{E^\bullet_x}=\partial_x\circ\Upsilon(x)+\Upsilon(x)\circ\partial_x.
\end{equation}
This suggests the possibility of completing the above definition of `homotopy representation' by adding the condition that there ought to be some section \[\Upsilon\in\Gamma\bigl(G_0;\End^{(-1)}(E^\bullet)\bigr)\] such that \eqref{eqn:2011a/5} is satisfied for all $x\in G_0$. (Note that $\Upsilon$ is not meant to become part of the defining data, which remain $E^\bullet$ and $\{R_k\}$.) This strengthening of the notion of homotopy representation is of course relevant to the problem addressed in the present note, because a homotopy representation that is unital up to homotopy will induce a true representation in cohomology. Starting from now, we use the term `homotopy representation' only in this stronger sense, in other words, we tacitly assume unitality up to homotopy.

\subsection*{A construction of homotopy representations}

\begin{lem}\label{lem:2011a/1} Let $G$ be a proper groupoid. Let $\lambda:s^*E\to t^*E$ be an arbitrary pseudo-representation of $G$ on a vector bundle $E$ over the base $G_0$ of $G$. Suppose that for some $G$-invariant open subset $U=GU\subset G_0$ the induced pseudo-representation $\lambda\mathbin|U$ of the restricted groupoid $G\mathbin|U\tto U$ on $E\mathbin|U$ is a representation (i.e.~respects units and composition). Then, for each point $x_0$ of $U$, there exists a homotopy representation $(E^\bullet,\{R_k\})$ of $G$ with the property that the induced isotropy representation on the degree zero cohomology group $H^0(E^\bullet_{x_0},\partial_{x_0})$ is equivalent to the isotropy representation $\lambda_{x_0}:G_{x_0}\to\GL(E_{x_0})$. Furthermore, whenever $\lambda$ is invertible, the same result can be achieved by means of a strongly invertible $(E^\bullet,\{R_k\})$. \end{lem}

\begin{proof} By properness, there exists some $G$-invariant continuous function $c:G_0\to[0,1]$ vanishing along the $G$-orbit of $x_0$ with $\supp(1-c)\subset U$; indeed, when $G$ is proper, the orbit space $G_0/G$ is locally compact Hausdorff, and the quotient projection $q:G_0\to G_0/G$ is open \cite[Lemme~6.5]{Tu99}, so for each point $p\in q(U)$ there exists some continuous function with value $1$ at $p$ whose support lies within $q(U)$. We shall regard $c$ as fixed throughout the proof.

Written out in components according to the grading, Equation~\eqref{eqn:2011a/4} reads
\begin{equation*}
 \partial_{tg}^{(n-1)}\circ\Psi^{(n)}(g,h)+\Psi^{(n+1)}(g,h)\circ\partial_{sh}^{(n)}=\lambda^{(n)}(g)\circ\lambda^{(n)}(h)-\lambda^{(n)}(gh).
\end{equation*}
In the special case of a two-term complex $E^\bullet=E^{(0)}\oplus E^{(1)}$, the only nontrivial such equations are the following two.
\begin{equation}
\label{eqn:2011a/6}
 \begin{gathered}
 \Psi^{(1)}(g,h)\circ\partial_{sh}^{(0)}=\lambda^{(0)}(g)\circ\lambda^{(0)}(h)-\lambda^{(0)}(gh)
\\%
 \partial_{tg}^{(0)}\circ\Psi^{(1)}(g,h)=\lambda^{(1)}(g)\circ\lambda^{(1)}(h)-\lambda^{(1)}(gh)
 \end{gathered}
\end{equation}
Moreover, Equation~\eqref{eqn:2011a/3} amounts to a single identity:
\begin{equation}
\label{eqn:2011a/7}
 \partial_{tg}^{(0)}\circ\lambda^{(0)}(g)=\lambda^{(1)}(g)\circ\partial_{sg}^{(0)}.
\end{equation}

So let us define our graded vector bundle $E^\bullet$ to be the direct sum of two copies of $E$ located in degrees zero and one. Define $\partial^{(n)}$ to be $c\cdot id_E$ for $n=0$ and zero for $n\neq0$. Put $\lambda^{(0)}=\lambda^{(1)}=\lambda$. The condition $0=\partial\circ\partial$ is of course satisfied. The identity \eqref{eqn:2011a/7} holds by the invariance of $c$. Again by the invariance of $c$, the two equations \eqref{eqn:2011a/6} reduce to the same identity, namely,
\begin{equation*}
 c(tg)\Psi^{(1)}(g,h)=\lambda(g)\circ\lambda(h)-\lambda(gh).
\end{equation*}
This suggests defining $\Psi^{(1)}:G_2\to\Hom(s_2^*E^{(1)},t_2^*E^{(0)})$ to be
\begin{equation*}
 \Psi^{(1)}(g,h)=%
\begin{cases}
 \text{zero}
&\text{if $g,h\in G\mathbin|U$}
\\
 \dfrac{1}{c(tg)}\bigl(\lambda(g)\circ\lambda(h)-\lambda(gh)\bigr)
&\text{if $g,h\in G\mathbin|\complement S$,}
\end{cases}
\end{equation*}
where $S=c^{-1}(0)\subset\supp(1-c)$ is a closed invariant subset of $G_0$ which is contained within $U$. Clearly $\Psi^{(1)}$ is well defined (because $\lambda\mathbin|U$ is a representation) and continuous.

The data which we have just introduced give rise to a homotopy representation provided they satisfy a few additional conditions. In detail, even though for any two-term homotopy representation one necessarily has $R_k=0$ for $k\geq3$, some of the so-called cocycle equations \eqref{eqn:2011a/1} survive, namely, for $k=3,4$:
\begin{gather*}
 R_2(g_1g_2,g_3)-R_2(g_1,g_2g_3)=R_1(g_1)\circ R_2(g_2,g_3)-R_2(g_1,g_2)\circ R_1(g_3)
\\%
 0=R_2(g_1,g_2)\circ R_2(g_3,g_4)
\end{gather*}
(for $k\geq5$ all cocycle equations are of course trivially satisfied). Spelled out in components, the latter condition reads \[0=\Psi^{(n-1)}(g_1,g_2)\circ\Psi^{(n)}(g_3,g_4),\] which is trivially satisfied since $\Psi^{(n)}=0$ for $n\neq1$. The former reduces to
\begin{equation*}
 \Psi^{(1)}(g_1g_2,g_3)-\Psi^{(1)}(g_1,g_2g_3)=\lambda(g_1)\circ\Psi^{(1)}(g_2,g_3)-\Psi^{(1)}(g_1,g_2)\circ\lambda(g_3).
\end{equation*}
For $g_1,g_2,g_3\in G\mathbin|\complement S$ this relation is true by definition of $\Psi^{(1)}$ and invariance of $c$; for $g_1,g_2,g_3\in G\mathbin|U$ it holds trivially. This concludes the proof of the lemma in the case when $\lambda$ is invertible.

When $\lambda$ is not invertible, it also becomes necessary to check unitality up to homotopy. We observe that Equation~\eqref{eqn:2011a/5} amounts to the identity
\begin{equation*}
 c(x)\Upsilon^{(1)}(x)=\lambda(1_x)-id.
\end{equation*}
We may then define $\Upsilon^{(1)}:G_0\to\End(E)$ to be zero when $x\in U$ and to be given by the preceding identity when $x\notin S$. \end{proof}

Our lemma immediately implies the existence of enough homotopy representations for any locally trivial bundle of compact groups.

\begin{prop}\label{prop:2011a/2} Let $G$ be a locally trivial bundle of compact groups. Every representation of the isotropy group $G_x$ at any base point $x$ can be realized as the degree-zero cohomology representation associated with some homotopy representation of $G$. \end{prop}

The situation appears more complicated when one restricts attention to strongly invertible homotopy representations.

\subsection*{The case of torus bundles}

\begin{lem}\label{lem:2011a/3} Let $\rho:\T^n\to\GU(k)$ be an arbitrary complex unitary representation of the $n$-torus. Then there exists some homotopy $H:\T^n\times[0,1]\to\SU(2k)$ with $H_t=%
\begin{bmatrix}
 \rho & 0
\\
    0 & \bar\rho
\end{bmatrix}$ for $t=1$ and with $H_t=id$ for $t=0$.\footnote{Notice that $\bar\rho$ coincides with the contragredient representation of $\rho$. Indeed, ${^t(\rho^{-1})}={^t(^t\bar\rho)}=\bar\rho$ since $\rho$ is a unitary matrix.} \end{lem}

\begin{proof} Let $\xi_1,\dotsc,\xi_k:\T^n\to\GU(1)$ be any characters on the $n$-torus. For every index $i=1,\dotsc,k$ let us write
\begin{equation*}
 \xi_i(a)=\xi_i(a_1,\dotsc,a_n)=\prod_{j=1}^n\xi_{ij}(a_j),
\end{equation*}
where $\xi_{ij}:\T^1\to\GU(1)$ denotes the homomorphism \[z\mapsto\xi_i(1,\dotsc,1,z,1,\dotsc,1)\] ($j$-th place). Let us consider the matrix representation $\T^n\to\SU(2k)$ given by
\begin{equation}
\label{eqn:2011a/8}
\begin{pmatrix}
\begin{smallmatrix}
 \xi_1(a) & 0
\\      0 & \bar\xi_1(a)
\\\cline{3-2}
\end{smallmatrix}
       \vline\mkern-9mu & \hdots & \smash{\hbox{\Large0}}
\\\cline{1-1}
                 \vdots & \ddots & \vdots
\\\cline{3-3}
 \smash{\hbox{\Large0}} & \hdots & \mkern-9mu\vline
\begin{smallmatrix}
\cline{3-2}
 \xi_k(a) & 0
\\      0 & \bar\xi_k(a)
\end{smallmatrix}
\end{pmatrix}=\prod_{j=1}^n%
\begin{pmatrix}
\begin{smallmatrix}
 \xi_{1j}(a_j) & 0
\\           0 & \bar\xi_{1j}(a_j)
\\\cline{3-2}
\end{smallmatrix}
       \vline\mkern-9mu & \hdots & \smash{\hbox{\Large0}}
\\\cline{1-1}
                 \vdots & \ddots & \vdots
\\\cline{3-3}
 \smash{\hbox{\Large0}} & \hdots & \mkern-9mu\vline
\begin{smallmatrix}
\cline{3-2}
 \xi_{kj}(a_j) & 0
\\           0 & \bar\xi_{kj}(a_j)
\end{smallmatrix}
\end{pmatrix}.
\end{equation}
Each factor of the matrix product on the right defines a homomorphism $\T^1\to\SU(2)\times\dotsb\times\SU(2)$ ($k$~factors). Since $\SU(2)$ is simply connected, we may find homotopies $H_{ij}:\T^1\times[0,1]\to\SU(2)$ contracting each block $\left(%
\begin{smallmatrix}
 \xi_{ij} & 0
\\      0 & \bar\xi_{ij}
\end{smallmatrix}\right)$ to the constant identity matrix $\left(%
\begin{smallmatrix}
 1 & 0
\\
 0 & 1
\end{smallmatrix}\right)$. Therefore, letting $pr_j:\T^n\to\T^1$ denote the projection on the $j$-th factor, the homotopy
\begin{equation*}
 H_t=\prod_{j=1}^n%
\begin{bmatrix}
 (H_{1j})_t\circ pr_j & \hdots & 0
\\
               \vdots & \ddots & \vdots
\\
                    0 & \hdots & (H_{kj})_t\circ pr_j
\end{bmatrix}
\end{equation*}
will contract \eqref{eqn:2011a/8} to the constant identity matrix within $\SU(2)\times\dotsb\times\SU(2)\subset\SU(2k)$. \end{proof}

The homotopy constructed in the lemma is simply a continuous path of (invertible) pseudo-representations since the map $H_t:\T^n\to\SU(2k)$ need not be a homomorphism of groups when $0<t<1$. The existence of enough strongly invertible homotopy representations on torus bundles is now an immediate consequence of our two lemmas. Actually, we can do slightly better. Let us call \emph{essentially abelian} any group whose identity component is abelian.

\begin{prop}\label{prop:2011a/4} Let $G$ be a locally trivial bundle of compact Lie groups whose fibers are essentially abelian. Then for each base point $x_0$ there exists some invertible complex pseudo-representation of $G$ with the property that its restriction over some open neighborhood of $x_0$ is a faithful representation. \end{prop}

\begin{proof} By hypothesis, there must be an open neighborhood $U$ of $x_0$ such that $G\mathbin|U\simeq K\times U$, for $K$ a compact Lie group with abelian identity component $T$. Choose any faithful complex unitary representation $\sigma:K\to\GU(N)$. Also choose representatives $k_1,\dotsc,k_\nu\in K$ for the various connected components $k_1T,\dotsc,k_\nu T$ of $K$ (left cosets of $T$ within $K$). Since $\GU(N)$ is connected, for each $i=1,\dotsc,\nu$ we can find a continuous path $P_i:[0,1]\to\GU(N)$ joining $\sigma(k_i)=P_i(1)$ and the $N$-by-$N$ identity matrix $I_N=P_i(0)$.

The Lie group $T$ is compact, connected and abelian, i.e., a torus. Lemma~\ref{lem:2011a/3} applies to the induced representation $\rho=\sigma\mathbin|T$ and tells us that there exists a continuous map $H:T\times[0,1]\to\SU(2N)$ such that $H_t=%
\begin{bmatrix}
 \rho & 0
\\
      0 & \bar\rho
\end{bmatrix}$ for $t=1$ and $H_t=I_{2N}$ for $t=0$. Now, let us define $A:K\times U\to\SU(2N)$ by setting
\begin{equation*}
 A(z,u)=%
\begin{bmatrix}
 P_z\bigl(c(u)\bigr) & 0
\\
                   0 & \bar P_z\bigl(c(u)\bigr)
\end{bmatrix}H\bigl(k_z^{-1}z,c(u)\bigr)
\end{equation*}
(matrix product), where $c:U\to[0,1]$ is any compactly supported continuous function such that $c=1$ in a neighborhood of $x_0$, and $P_z$ resp.~$k_z$ denote $P_i$ resp.~$k_i$ for the unique index $i$ such that $z\in k_iT$. Then $A$ is continuous, of constant value $%
\begin{bmatrix}
 \sigma & 0
\\
      0 & \bar\sigma
\end{bmatrix}$ on some open neighborhood of $x_0$ in $U$, and of constant value $I_{2N}$ outside some compact neighborhood of $x_0$ in $U$. We define our pseudo-representation to be equal to $A$ on $G\mathbin|U\simeq K\times U$ and trivial outside $U$. \end{proof}

\begin{cor}\label{cor:2011a/5} On any locally trivial bundle of compact Lie groups with essentially abelian fibers, there exist for each base point $x_0$ strongly invertible homotopy representations with the property that the isotropy representations at $x_0$ which they induce in cohomology are faithful. \end{cor}

For a generic torus bundle, it may happen that not every faithful representation of one fiber can be extended to an invertible pseudo-representation of the whole bundle. For instance, let $G$ be the $\T^2$-bundle over $S^1$ (the circle) arising as the quotient of the trivial $\T^2$-bundle $\T^2\times\R\to\R$ modulo the equivalence
\begin{equation}
\label{eqn:2011a/9}
 (a,b;t)\sim(a,a^lb;t+l)\makebox[.em][l]{\qquad($l\in\Z$)}
\end{equation}
(cf.~\cite[Example 2.10]{Tre10}). \em A given representation $\T^2\to\GU(k)$ may be extended to an invertible pseudo-representation of $G$ only providing that its restriction to the subgroup $1\times\T^1\subset\T^2$ lies within $\SU(k)$\em.

Indeed, suppose $\lambda:G\to\GL(E)$ is an invertible pseudo-representation of $G$ on some vector bundle $E$ over $S^1$. Its determinant $\delta=\det\lambda:G\to\C^\times$ is a continuous function into the multiplicative group of all nonzero complex numbers. We claim that the restriction of $\delta$ to the $\T^1$-subbundle of $G$ consisting of all pairs of the form $(1,b;t)$ in $\T^2\times\R$ must be a fiberwise contractible map; evidently, this will imply the desired statement about representations $\T^2\to\GU(k)$.

Regard $\delta$ as a map $\T^2\times\R\to\C^\times$ compatible with the equivalence relation \eqref{eqn:2011a/9}. The two maps $\delta_0,\delta_1:\T^2\to\C^\times$ given by $\delta_i(a,b)=\delta(a,b;i)$~[$i=0,1$] are homotopic. In particular, the two maps $\T^1\to\C^\times$ given by $a\mapsto\delta_0(a,1)$ and $a\mapsto\delta_1(a,1)$ are homotopic. By \eqref{eqn:2011a/9}, the former map coincides with $a\mapsto\delta_1(a,a)$. Now the loop in $\T^2$ given by $a\mapsto(a,a)$ is homotopic to the concatenation of the two loops $a\mapsto(a,1)$ and $a\mapsto(1,a)$; the loop $a\mapsto\delta_1(a,a)$ in $\C^\times$ is thus homotopic to the concatenation of $a\mapsto\delta_1(a,1)$ with $a\mapsto\delta_1(1,a)$. Applying the degree homomorphism $\deg:\pi_1\bigl(\C^\times,\delta_1(1,1)\bigr)\to\Z$ to these loops, we obtain
\begin{equation*}
 \deg\delta_1(-,1)=\deg\delta_1(-,1)+\deg\delta_1(1,-).
\end{equation*}
The loop $b\mapsto\delta_1(1,b)$ in $\C^\times$ must therefore be contractible.

\subsection*{The derived representation category of a compact group}

It is possible that some of the results contained in this section be already known to experts. Anyway, since we have not been able to find any references in the literature, we believe it appropriate to provide the reader with a self-contained exposition, relying on minimal background.

\begin{lem}\label{lem:2015b/6} Let $(E^\bullet,\{R_k\})$ and $(F^\bullet,\{S_k\})$ be two homotopy representations of a proper groupoid $G\tto G_0$. Suppose that $R_0=0$ and that $S_0=0$, so that $R_1$ and $S_1$ are honest (graded) representations of $G$. Let $\varPhi_0:E^\bullet\to F^\bullet$ be an equivariant vector bundle morphism of degree zero. Then, there exists a homomorphism of homotopy representations $\varPhi=\{\varPhi_k\}:(E^\bullet,\{R_k\})\to(F^\bullet,\{S_k\})$ which extends $\varPhi_0$. \end{lem}

\begin{proof} Let us put $\lambda=R_1$ and $\mu=S_1$ for brevity. We are looking for vector bundle morphisms $\varPhi_k:s_{k}^*E^\bullet\to t_{k}^*F^\bullet$ of degree $-k$, one for each integer $k\geq1$, such that the following identities are satisfied \cite[p.~431]{AC13}:
\begin{multline}
\label{eqn:2015b/10}
 \sum_{i+j=k}(-1)^j\varPhi_j(g_1,\dotsc,g_j)\circ R_i(g_{j+1},\dotsc,g_k)=\sum_{i+j=k}S_j(g_1,\dotsc,g_j)\circ{}%
\\[-.5\baselineskip]%
 \varPhi_i(g_{j+1},\dotsc,g_k)+\sum_{j=1}^{k-1}(-1)^j\varPhi_{k-1}(g_1,\dotsc,g_jg_{j+1},\dotsc,g_k).
\end{multline}

We argue inductively. Suppose that for some integer $l\geq0$ we are given a complete sequence of solutions $\varPhi_0,\dotsc,\varPhi_l$ to Equation~\eqref{eqn:2015b/10} for every value of $k$ between zero and $l+1$. We contend that $\varPhi_{l+1}$ defined by the integral formula
\begin{multline}
\label{eqn:2015b/11}
 \varPhi_{l+1}(g_1,\dotsc,g_{l+1})=\integral_{tg=sg_{l+1}}(-1)^l\biggl\{\sum_{j=0}^l(-1)^j\varPhi_j(g_1,\dotsc,g_j)\circ{}%
\\%
\shoveright{%
 R_{l+2-j}(g_{j+1},\dotsc,g_{l+1},g)-{}%
}\\%
 \sum_{j=2}^{l+2}S_j(g_1,\dotsc,g_j)\circ\varPhi_{l+2-j}(g_{j+1},\dotsc,g_{l+1},g)\biggr\}\circ\lambda(g)^{-1}dg
\end{multline}
is a solution to Equation~\eqref{eqn:2015b/10} for $k=l+2$.

We shall do the computation just for $l=0$, by way of example. The computation for higher values of $l$, similar but more involved, will be left to the reader's patience. (We recommend doing the cases $l=1,2$ as a preparatory exercise.) For $l=0$, our formula \eqref{eqn:2015b/11} reduces to the following expression for $\varPhi_1$:
\begin{equation*}
 \varPhi_1(g_1)=\integral_{tg=sg_1}\{\varPhi_0(tg_1)\circ R_2(g_1,g)-S_2(g_1,g)\circ\varPhi_0(sg)\}\circ\lambda(g)^{-1}dg.
\end{equation*}
By making use of the circumstance that $\lambda$ is an honest representation, of the left invariance of the chosen Haar measure system on $G$, of the $G$-equivariance of $\varPhi_0$, and of the cocycle identities \eqref{eqn:2011a/1} satisfied by $R_2$ and $S_2$, we obtain
\begin{align*}
&\mu(g_1)\circ\varPhi_1(g_2)-\varPhi_1(g_1g_2)+\varPhi_1(g_1)\circ\lambda(g_2)
\\*
&\hskip.5em=\varint\{\mkern-3mu%
 \begin{aligned}[t]
 \mu(g_1)\circ[\varPhi_0(tg_2)\circ R_2(g_2,g)%
&-S_2(g_2,g)\circ\varPhi_0(sg)]%
 \\
 {}-\varPhi_0(tg_1)\circ R_2(g_1g_2,g)%
&+S_2(g_1g_2,g)\circ\varPhi_0(sg)%
 \\
 {}+\varPhi_0(tg_1)\circ R_2(g_1,g_2g)%
&-S_2(g_1,g_2g)\circ\varPhi_0(sg)\}\circ\lambda(g)^{-1}\mkern\thinmuskip\d g
 \end{aligned}
\\%
&\hskip.5em=\varint\{\mkern-3mu%
 \begin{aligned}[t]
&\varPhi_0(tg_1)\circ[\lambda(g_1)\circ R_2(g_2,g)-R_2(g_1g_2,g)+R_2(g_1,g_2g)]%
 \\
&-[\mu(g_1)\circ S_2(g_2,g)-S_2(g_1g_2,g)+S_2(g_1,g_2g)]\circ\varPhi_0(sg)\}\circ\lambda(g)^{-1}\mkern\thinmuskip\d g
 \end{aligned}
\\%
&\hskip.5em=\integral\{\varPhi_0(tg_1)\circ R_2(g_1,g_2)\circ\lambda(g)-S_2(g_1,g_2)\circ\mu(g)\circ\varPhi_0(sg)\}\circ\lambda(g)^{-1}dg
\\\intertext{%
whence, remembering that we are using a normalized Haar measure system, and making use of the $G$-equivariance of $\varPhi_0$ one more time,
}%
&\hskip.5em=\varPhi_0(tg_1)\circ R_2(g_1,g_2)-S_2(g_1,g_2)\circ\varPhi_0(sg_2),
\end{align*}
which is Equation~\eqref{eqn:2015b/10} for $k=2$, as contended. \end{proof}

We shall now digress briefly to recall some basic constructions concerning DG categories \cite{Kel94}. We remind the reader that in a DG category $\catfont A$ the hom-set between any two objects $A,B$ is a graded vector space \[\catfont A(A,B)^\bullet=\textstyle\bigoplus\limits_{n\in\Z}\catfont A(A,B)^{(n)}\] endowed with a linear operator $d_{A,B}$ of degree $+1$ such that $d_{A,B}\circ d_{A,B}=0$. The composition law in $\catfont A$ is required to be compatible with the grading and with the differentials: for $g\in\catfont A(B,C)$ homogeneous and for $f\in\catfont A(A,B)$,
\begin{equation*}
 d_{A,C}(gf)=(d_{B,C}g)f+(-1)^{\deg g}g(d_{A,B}f).
\end{equation*}
One also assumes that for every object $A$ the identity morphism $1_A$ is homogeneous of degree zero. It follows that $d_{A,A}(1_A)=0$.

For an arbitrary DG category $\catfont A$, one can consider the subcategory $\Zcat\catfont A$ consisting of all those (mixed degree) morphisms $f\in\catfont A(A,B)$ such that \[d_{A,B}f=0.\] Notice that if one expresses a morphism $f$ uniquely as the sum of its homogeneous components $\sum f^{(n)}$, then $f$ lies within $\Zcat\catfont A$ if and only if so does every $f^{(n)}$. The grading on $\catfont A$ therefore induces a grading on $\Zcat\catfont A$. Within $\Zcat\catfont A$, the homogeneous morphisms of degree zero constitute a subcategory $\Zcat^0\catfont A$. For each pair of objects $A,B$, the equivalence relation on the hom-set $\Zcat\catfont A(A,B)$
\begin{equation*}
 f_1\sim f_2\aeq f_2-f_1=d_{A,B}h\text{ \it for some $h\in\catfont A(A,B)$}
\end{equation*}
is compatible with the grading in the sense that $f_1\sim f_2$ if, and only if, $f_1^{(n)}\sim f_2^{(n)}$ for all $n$. This equivalence relation on the set of morphisms of the category $\Zcat\catfont A$ is also compatible with composition, so it is a categorical congruence. The resulting quotient category, which we designate $\Hcat\catfont A$, has hom-sets \[\Hcat\catfont A(A,B)=\Zcat\catfont A(A,B)/\mathord\sim.\] The grading on $\Zcat\catfont A$ descends to a grading on $\Hcat\catfont A$. We let $\Hcat^0\catfont A$ denote the subcategory of $\Hcat\catfont A$ formed by all degree zero morphisms. One defines the \emph{null-homotopic} morphisms $f\in\Zcat\catfont A(A,B)$ to be those of the form $f=d_{A,B}h$.

The homotopy representations of a groupoid $G$ naturally form a DG category \cite[Remark 3.8]{AC13}, here denoted $\HRep(G)$. By definition, a degree~$n$ morphism \[\varPhi\in\HRep(G)(E,F)^{(n)}\] between two homotopy representations $E=(E^\bullet,\{R_k\})$ and $F=(F^\bullet,\{S_k\})$ is a $C^\bullet(G)$-linear map of degree $n$ between the graded $C^\bullet(G)$-module $C(G,E)^\bullet$ and the graded $C^\bullet(G)$-module $C(G,F)^\bullet$. The differential of $\varPhi$ is defined to be
\begin{equation*}
 D_{E,F}\varPhi=D_F\circ\varPhi-(-1)^{\deg\varPhi}\varPhi\circ D_E,
\end{equation*}
where $D_E$ and $D_F$ denote the differential operators on $C(G,E)^\bullet$ and $C(G,F)^\bullet$ associated to $E$ and $F$, respectively. In the notations of \cite[Lemma~3.10]{AC13}, the effect of $\varPhi$ on any cochain $\eta\in C(G;E)$ can be written in the form
\begin{equation*}
 \varPhi(\eta)=\varPhi_0\star\eta+\varPhi_1\star\eta+\varPhi_2\star\eta+\dotsb,
\end{equation*}
each $\varPhi_k$ being a vector bundle morphism of $s_{k}^*E^\bullet$ into $t_{k}^*F^\bullet$ of degree $-k+n$. The following expression for the differential $D_{E,F}\varPhi$ can then be obtained by means of calculations analogous to those that in \cite{AC13} yield the identities \eqref{eqn:2015b/10}:
\begin{multline}
\label{eqn:2015b/12}
 (D_{E,F}\varPhi)_k(g_1,\dotsc,g_k)=\sum_{i+j=k}(-1)^{j\deg\varPhi}S_j(g_1,\dotsc,g_j)\circ\varPhi_i(g_{j+1},\dotsc,g_k)-{}%
\\%
\shoveright{%
 (-1)^{\deg\varPhi}\sum_{i+j=k}(-1)^j\varPhi_j(g_1,\dotsc,g_j)\circ R_i(g_{j+1},\dotsc,g_k)+{}%
}\\[-.5\baselineskip]%
 (-1)^{\deg\varPhi}\sum_{j=1}^{k-1}(-1)^j\varPhi_{k-1}(g_1,\dotsc,g_jg_{j+1},\dotsc,g_k).
\end{multline}
The notion of homomorphism of homotopy representations, which is characterized by the identities \eqref{eqn:2015b/10}, corresponds to the notion of morphism within the subcategory $\Zcat^0\HRep(G)\subset\HRep(G)$.

\begin{lem}\label{lem:2015b/7} Let $G\tto G_0$ be a proper groupoid, and let $E=(E^\bullet,\{R_k\})$ and $F=(F^\bullet,\{S_k\})$ be two homotopy representations of $G$ such that $R_0=0$ and $S_0=0$. Let $\varPhi:C(G,E)^\bullet\to C(G,F)^{\bullet+n}$ be a degree~$n$ morphism in $\HRep(G)$ between them, with $D_{E,F}\varPhi=0$. Suppose that $\varPhi_0=0$. Then $\varPhi$ is null homotopic. \end{lem}

\begin{proof} We shall give a proof under the simplifying assumption that $R_k$ and $S_k$ be equal to zero for $k\neq1$, which is the only case we are really interested in. The general case will be left to the reader. Let us set $\lambda=R_1$ and $\mu=S_1$. In view of \eqref{eqn:2015b/12}, we need to solve for $\varPsi_k$ in the following equation, for each value of $k=0,1,2,\dotsc$:
\begin{multline}
\label{eqn:2015b/13}
 \varPhi_{k+1}(g_1,\dotsc,g_{k+1})=(-1)^{n-1}\biggl[\mu(g_1)\circ\varPsi_k(g_2,\dotsc,g_{k+1})-{}%
\\[-.4\baselineskip]%
 (-1)^k\varPsi_k(g_1,\dotsc,g_k)\circ\lambda(g_{k+1})+\sum_{j=1}^k(-1)^j\varPsi_k(g_1,\dotsc,g_jg_{j+1},\dotsc,g_{k+1})\biggr].
\end{multline}
Using the equation $D_{E,F}\varPhi=0$ satisfied by $\varPhi$, it is easy to check that
\begin{equation*}
 \varPsi_k(g_1,\dotsc,g_k)=(-1)^k\integral_{tg=sg_k}\varPhi_{k+1}(g_1,\dotsc,g_k,g)\circ\lambda(g)^{-1}dg
\end{equation*}
must be a solution to \eqref{eqn:2015b/13}. \end{proof}

We shall make use of the following abbreviations.
\begin{gather*}
 \DRep(G)=\Hcat\bigl(\HRep(G)\bigr)
\\%
 \DRep^0(G)=\Hcat^0\bigl(\HRep(G)\bigr)
\end{gather*}
(Cf.~the discussion following Definition~3.9 and 3.27 in \cite{AC13}.) Suppose that $G$ is a group. In this case we have a \emph{cohomology representation functor}
\begin{equation}
\label{eqn:2015b/14}
 R:\Zcat\HRep(G)\longto\Rep(G)
\end{equation}
which to each homotopy representation $E$ associates the corresponding cohomology representation $\bigoplus H^i(E^\bullet,\partial_E)$ and to each degree~$n$ morphism \[\varPhi:C(G,E)^\bullet\longto C(G,F)^{\bullet+n}\] satisfying the condition $D_{E,F}\varPhi=0$ associates the homomorphism of cohomology representations that is induced by $\varPhi_0$ in virtue of the equations \eqref{eqn:2015b/12} for $k=0,1$. The same equations show that if $\varPhi=D_{E,F}\varPsi$ for some $\varPsi$ of degree $n-1$ then $\varPhi_0$ must induce the zero map in cohomology. The functor \eqref{eqn:2015b/14} therefore factors through the homotopy category $\DRep(G)$. We refer to the resulting functor of $\DRep(G)$ into $\Rep(G)$ still as the cohomology representation functor $R$.

There is a tensor product on $\DRep(G)$, canonical up to natural equivalence of symmetric monoidal categories, which makes $R$ into a monoidal functor, and which is induced by a corresponding ``$\infty$-monoidal structure'' on the DG category $\HRep(G)$. We refer the reader to \cite{ACD11} for details.

\begin{prop}\label{prop:2015b/8} For any compact group $G$, the cohomology representation functor
\begin{equation}
\label{eqn:2015b/15}
 R:\DRep(G)\longto\Rep(G)
\end{equation}
is an equivalence of (symmetric monoidal) categories. \end{prop}

\begin{proof} Since every ordinary representation of $G$ can be regarded as an object of the category $\HRep(G)$, the cohomology representation functor is trivially surjective on objects. Next, by Theorem~3.32 of \cite{AC13}, any object of $\HRep(G)$ is quasi isomorphic\textemdash equivalently, isomorphic within the homotopy category \[\DRep^0(G)\subset\DRep(G)\]\cite[Proposition~3.28]{AC13}\textemdash to a homotopy representation which has zero differential. By our Lemma~\ref{lem:2015b/6}, in combination with Proposition~3.28 of loc.cit., any such representation must in turn be quasi isomorphic to an ordinary graded representation (namely the associated cohomology representation). Hence, in order to show that the functor \eqref{eqn:2015b/15} is fully faithful, it will be enough to show that such is its restriction to the full subcategory whose objects are the ordinary graded representations. An arbitrary morphism $\varPhi$ between any two such objects uniquely decomposes as $\varPhi=\bigoplus\varPhi^{(n)}$, with each $\varPhi^{(n)}$ homogeneous of degree $n$, and one has $R(\varPhi)=\bigoplus R(\varPhi^{(n)})$. Faithfulness is then an immediate consequence of our Lemma~\ref{lem:2015b/7}. Fullness is obvious. \end{proof}

The same proof also shows that there is an equivalence of categories
\begin{equation*}
 \DRep^0(G)\simeq\Rep^{(\Z)}(G),
\end{equation*}
where $\Rep^{(\Z)}(G)$ stands for the category of $\Z$-graded representations of $G$. Observe that there is a canonical equivalence of symmetric tensor categories
\begin{equation*}
 \Rep^{(\Z)}(G)\simeq\Rep\bigl(\GU(1)\times G\bigr).
\end{equation*}
This can easily be understood in terms of the tannakian formalism as follows. (Compare \cite[p.~186, (5.1a)\textendash(5.1d)]{DM82}.) By a \emph{tannakian $\Z$-grading} on an additive complex symmetric tensor category with conjugation $(\catfont C,\otimes)$ we mean a direct sum decomposition of each object $X$ of $\catfont C$ \[X=\textstyle\bigoplus\limits_{n\in\Z}X^{(n)}\] which is functorial in $X$ and compatible with the tensor product \[(X\otimes Y)^{(n)}=\textstyle\bigoplus\limits_{r+s=n}X^{(r)}\otimes Y^{(s)}\] and with the conjugation. For an arbitrary compact group $G$, there is a one-to-one correspondence between tannakian $\Z$-gradings on the representation category $\Rep(G)$ and central homomorphisms $\varepsilon:\GU(1)\to G$. Namely, given $\varepsilon$, for each representation $\rho:G\to\GL(V)$ one defines $V^{(n)}$ to be the $n$-th isotypical summand for the $\GU(1)$-representation $\rho\circ\varepsilon:\GU(1)\to\GL(V)$ [that is, the subspace where $\GU(1)$ acts by the complex character $z\mapsto z^n$] and then $\rho^{(n)}$ to be the representation induced by $\rho$ on the $G$-invariant subspace $V^{(n)}\subset V$.

The fundamental importance of $\GU(1)$~central extensions in quantum physics is stressed in \cite{TW87}.

\begingroup
\let\bibsection\subsection
\def\section{\bibsection}

\endgroup
\end{document}